\documentclass[12pt,A4]{article}
\usepackage[english]{babel}
\usepackage{amsmath,amsfonts,amssymb,amsthm,mathrsfs,bbm}
\usepackage[font=sf, labelfont={sf,bf}, margin=1cm]{caption}
\usepackage{graphicx,graphics}
\usepackage{epsfig}
\usepackage{latexsym}
\usepackage[applemac]{inputenc}
\usepackage{ae,aecompl}
\usepackage{pstricks}
\usepackage{enumerate}
\usepackage{xcolor}
\usepackage[pdfpagemode=UseNone,bookmarksopen=false,colorlinks=true,urlcolor=blue,citecolor=blue,citebordercolor=blue,linkcolor=blue]{hyperref}

\usepackage[normalem]{ulem}

\usepackage{comment}

\usepackage[all, knot, cmtip]{xy}

\newcommand{\ndN}{\mathbb{N}}

\newcommand{\ndR}{\mathbb{R}}

\renewcommand{\Pr}[1]{\mathbb{P}(#1)}
\newcommand{\Prb}[1]{\mathbb{P}\left( #1 \right)}
\newcommand{\Ex}[1]{\mathbb{E}[#1]}
\newcommand{\Exb}[1]{\mathbb{E}\left[ #1 \right]}

\newcommand{\cF}{\mathcal{F}}

\newcommand{\cG}{\mathcal{G}}

\newcommand{\cH}{\mathcal{H}}

\newcommand{\mF}{\mathsf{F}}

\newcommand{\mG}{\mathsf{G}}

\newcommand{\mH}{\mathsf{H}}

\newcommand{\mS}{\mathsf{S}}
\newcommand{\mR}{\mathsf{R}}

\newcommand{\Sym}{\text{Sym}}

\newcommand{\cE}{\mathcal{E}}

\newcommand{\eqdist}{\,{\buildrel d \over =}\,}

\newcommand{\Set}{\textsc{SET}}

\newtheorem{theorem}{Theorem}[section]

\newtheorem{corollary}[theorem]{Corollary}

\newtheorem{lemma}[theorem]{Lemma}

\newtheorem{definition}[theorem]{Definition}

\numberwithin{equation}{section}

\title{\textbf{Unlabelled Gibbs partitions}}
\date{}

\author{Benedikt Stufler\thanks{\'Ecole Normale Sup\'erieure de Lyon, E-mail: benedikt.stufler@ens-lyon.fr; The author is supported by the German Research Foundation DFG, STU 679/1-1}}

\begin{document}
	
	\maketitle
	
\let\thefootnote\relax\footnotetext{ \\\emph{MSC2010 subject classifications}. Primary  60C05, 05A18; secondary 60B10. }

\vspace {-0.5cm}

\begin{abstract}
We study random composite structures considered up to symmetry that are sampled according to weights on the inner and outer structures.  This model may be viewed as an unlabelled version of Gibbs partitions and encompasses multisets of weighted combinatorial objects. We describe a general setting characterized by the formation of a giant component. The collection of small fragments is shown to converge in total variation toward a limit object following a P\'olya-Boltzmann distribution.
\end{abstract}

\section{Introduction}
The study of the evolution of shapes of random ensembles, as the total size becomes large, has a long history, and connections to a variety of  fields sucht statistical mechanics, representation theory, and combinatorics are known. A sketch of the history of limit shapes may be found in the work by Erlihson and Granovsky~\cite{MR2453776} on Gibbs partitions in the expansive case, and we refer the reader to this informative summary and references given therein for an adequate treatment of the historical development.

The term Gibbs partitions was coined by Pitman~\cite{MR2245368} in his comprehensive survey on combinatorial stochastic processes. It describes a model of random partitions of sets, where the collection of classes as well as each individual partition class are endowed with a weighted structure.  For example, in simply generated random plane forest, each component is endowed with a tree structure carrying a non-negative weight, and the collection of components carries a linear order. Likewise, Gibbs partitions also encompass various of types of random graphs whose vertex sets are partitioned by their connected components.

Many structures such as classes of graphs may also be viewed up to symmetry. The symmetric group acts in a canonical way on the collection of composite structures over a fixed set, and its orbits are called unlabelled objects. Sampling such an isomorphism class with probability proportional to its weight is the natural unlabelled version of the Gibbs partition model. This encompasses as a special case the important model of random multisets, which has been studied by Bell, Bender, Cameron and Richmond~\cite{MR1763972}, and which is also encompassed in the setting by Arratia, Barbour and Tavar{\'e}~\cite{MR2032426} and Barbour and Granovsky~\cite{MR2121024}. The important example of forests of unlabelled trees has been considered by Mutafchiev~\cite{MR1662783}. General unlabelled Gibbs partitions, however, appear to have not received any attention in the literature so far. This is possibly due to the fact that this model of random ensembles is quite involved, as the symmetries of both the inner and outer structures influence its behaviour. This makes it particularly hard to arrive at general results that characterize the asymptotic behaviour for a wide range of species of structures. Nevertheless, it is natural to consider combinatorial objects up symmetry, and to ask, whether similar regimes such as for example the expansive case \cite{MR2453776} or the convergent case \cite{Mreplaceme} may also be found in the unlabelled setting.

For this reason, the present work aims make a first step in this direction, with the hope that this may incite further research. We study a general setting characterized by the formation of a giant component with a stochastically bounded remainder.  This phenomenon may for example be observed for uniformly sampled unordered forests of unlabelled trees as the total number of vertices tends to infinity, regardless whether we consider trees that are rooted or unrooted, ordered or unordered. The small fragments are shown to converge in total variation towards a limit object following a Pólya-Boltzmann distribution, a term coined by Bodirsky, Fusy, Kang and Vigerske~\cite{MR2810913}, who generalized and further developed the theory of Boltzmann samplers initiated in \cite{MR2095975,MR2498128}. 
Rather than taking a pure generating function viewpoint, our approach is to use the methods from \cite{MR2810913} to reduce each problem to probabilistic questions. This allows us to prove our results in great generality and economically make use of available results for heavy-tailed and subexponential probability distributions~\cite{MR3097424, MR772907,MR714482, MR0348393}.

The present work is also the logical continuation of \cite{Mreplaceme}, where a gelation phenomenon was observed for labelled Gibbs partitions. The P\'olya-Boltzmann sampler framework of \cite{MR2810913} allows us to pursue a similar overall strategy as in \cite{Mreplaceme}, but our proofs are more involved and technical, as we have to consider objects up to symmetry.

 The motivation of this particular  line of research stems from the study of random graphs from restricted classes. McDiarmid~\cite{MR2418771, MR2507738} showed that the small fragments of a random graph from a minor-closed addable class converge toward a Boltzmann Poisson random graph. In this work, McDiarmid poses the question, if a similar behaviour may be observed for unlabelled graphs.  As was shown in \cite{Mreplaceme}, an approach via Gibbs partitions and conditioned Galton--Watson trees is possible in the labelled setting. Hence it is natural to ask, whether a similar strategy also works in the unlabelled setting. The present work provides a first piece to the puzzle, and we hope to pursue this question further in future work.

\subsection*{Plan of the paper}
In Section~\ref{sec:prel} we fix notations and recall some background related to Gibbs partitions, combinatorial species, P\'olya-Boltzmann distributions and subexponential sequences. Section~\ref{sec:conv} presents our main results for unlabelled Gibbs partitions. In Section~\ref{sec:proofs} we collect all proofs.

\section{Preliminaries}
\label{sec:prel}

\subsection{Notation}
We use the notation
\[
\ndN=\{1,2,\ldots\}, \qquad \ndN_0 = \{0\} \cup \ndN, \qquad [n]=\{1,2,\ldots, n\}, \qquad n \in \ndN_0,
\]
and let $\ndR_{>0}$ and $\ndR_{\ge 0}$ denote the sets of positive and non-negative real numbers, respectively. 
Throughout, we assume that all considered random variables are defined on a common probability space $(\Omega, \mathscr{F}, \mathbb{P})$. 
All unspecified limits are taken as $n$ becomes large, possibly along an infinite subset of $\ndN$.

A function $h: \ndR_{>0} \to \ndR_{>0}$ is termed {\em slowly varying}, if for any fixed $t >0$ it holds that
\[
\lim_{x \to \infty}\frac{h(tx)}{h(x)} = 1.
\]
For any power series $f(z) = \sum_n f_n z^n$, we let $[z^n]f(z) = f_n$ denote the coefficient of $z^n$. A sequence of $\ndR$-valued random variables $(X_n)_{n \ge 1}$ is {\em stochastically bounded}, if for each $\epsilon > 0$ there is a constant $M>0$ with
\[
\limsup_{n \to \infty} \Pr{ |X_n| \ge M} \le \epsilon.
\]
The {\em total variation distance} between two random variables $X$ and $Y$ with values in a countable state space $S$ is defined by
\[
d_{\textsc{TV}}(X,Y) = \sup_{\cE \subset S} |\Pr{X \in \cE} - \Pr{Y \in \cE}|.
\]

\subsection{Weighted combinatorial species and cycle index sums}
The present section recalls the necessary species-theory following Joyal~\cite{MR633783}. A {\em species of combinatorial structures} $\cF^\omega$  with non-negative weights is a functor that produces for each finite set $U$ a finite set $\cF[U]$ of {\em $\cF$-structures} and a map
\[
\omega_U: \cF[U] \to \ndR_{\ge 0}.
\]
We will often write $\omega(F)$ instead of $\omega_U(F)$ for the weight of a structure $F \in \cF[U]$. If no weighting is specified explicitly, we assume that any structure receives weight $1$. We refer to the set $U$ as the set of {\em labels} or {\em atoms} of the structure. For any $\cF$-object $F \in \cF[U]$ we let
$
|F| := |U| \in \ndN_0
$
denote its {\em size}. The species $\cF$ is further required to produce for each bijection $\sigma: U \to V$ a corresponding bijection 
\[
\cF[\sigma]: \cF[U] \to \cF[V]
\]
that preserves the $\omega$-weights. In other words, the following diagram must commute.
\[
\xymatrix{ \cF[U]  \ar[r]^{\cF[\sigma]} \ar[dr]^{\omega_U} 
	&\cF[V]\ar[d]^{\omega_V}\\
 		    &\ndR_{\ge 0}}
\]
Species are also subject to the usual functoriality requirements: the identity map $\text{id}_U$ on $U$ gets mapped to the identity map $\cF[\text{id}_U] = \text{id}_{\cF[U]}$ on the set $\cF[U]$. For any bijections $\sigma: U \to V$ and $\tau: V \to W$ the diagram
\[
\xymatrix{ \cF[U]  \ar[r]^{\cF[\sigma]} \ar[dr]^{\cF[\tau \sigma]} 
	&\cF[V]\ar[d]^{\cF[\tau]}\\
	&\cF[W]}
\]
commutes. We further assume that $\cF[U] \cap \cF[V] = \emptyset$ whenever $U \ne V$. This is not much of a restriction, as we may always replace $\cF[U]$ by $\{U\} \times \cF[U]$ for all sets $U$, to make sure that it is satisfied.

Two weighted species $\cF^\omega$ and $\cH^\gamma$ are said to be {\em structurally equivalent} or {\em isomorphic}, denoted by $\cF^\omega \simeq \cH^\gamma$, if there is a family of weight-preserving bijections $(\alpha_U: \cF[U] \to \cH[U])_U$ with $U$ ranging over all finite sets, such the following diagram commutes for each  bijection  $\sigma: U \to V$ of finite sets.
\[
\xymatrix{ \cF[U] \ar[d]^{\alpha_U} \ar[r]^{\cF[\sigma]} &\cF[V]\ar[d]^{\alpha_V}\\
	\cH[U] \ar[r]^{\cG[\sigma]} 		    &\cH[V]}
\]

For any finite set $U$, the symmetric group $\mathscr{S}_U$ acts on the set $U$ via
\[
	\sigma.F = \cF[\sigma](F)
\]
for all $F \in \cF[U]$ and $\sigma \in \mathscr{S}_U$. A bijection $\sigma$ with $\sigma.F = F$ is termed an {\em automorphism} of $F$. We let $\tilde{\cF}[U]$ denote the orbits of this group action. All $\cF$-objects of an orbit $\tilde{F}$ have the same size and same $\omega$-weight, which we denote by $|\tilde{F}|$ and $\omega(\tilde{F})$. It will be convenient to use the notation
\[
\mathscr{U}(\cF) = \bigcup_{k \ge 0} \mathscr{U}_k (\cF) \qquad \text{with} \qquad \mathscr{U}_k (\cF) = \tilde{\cF}[k].
\]
Formally, an {\em unlabelled} $\cF$-object is defined as an isomorphism class of $\cF$-objects. We may also identify the unlabelled objects of a given size $n$ with the orbits of the action of the symmetric group on any $n$-sized set. In particular, the collection of unlabelled $\cF$-objects may be identified with the set $\mathscr{U}(\cF)$. By abuse of notation, we treat unlabelled objects as if they were regular $\cF$-objects. The power series
\[
	\tilde{\cF}^\omega(z) = \sum_{\tilde{F} \in \mathscr{U}(\cF)} \omega(\tilde{F}) z^{|\tilde{F}|}.
\]
is the {\em ordinary generating series} of the species.

To any species $\cF$ we may assign the corresponding species $\Sym(\cF)$ of {\em $\cF$-symmetries} such that
\[
	\Sym(\cF)[U] = \{ (F, \sigma) \mid F \in \cF[U], \sigma \in \mathscr{U}, \sigma.F = F\}.
\]
In other words, a symmetry is a pair of an $\cF$-object and an automorphism. 
The transport along a bijection $\gamma: U \to V$ is given by
\[
	\Sym(\cF)[\gamma](F, \sigma) = (\cF[\gamma](F), \gamma \sigma \gamma^{-1}).
\]
For any permutation $\sigma$ we let $\sigma_i$ denote its number of $i$-cycles. The {\em cycle index series} of a species $\cF$ is defined as the formal power series
\[
	Z_{\cF^\omega}(z_1, z_2, \ldots) = \sum_{k \ge 0} \sum_{(F, \sigma) \in \Sym(\cF)[k]} \frac{\omega(F)}{k!} z_1^{\sigma_1} \cdots z_k^{\sigma_k}
\]
in countably infinitely many indeterminates $z_1, z_2, \ldots$. The following standard result is given for example by Bergeron, Labelle and Leroux \cite[Ch. 2.3]{MR1629341} and shows how the ordinary generating series and the cycle index sum of a species are related.
\begin{lemma}
	\label{le:relation}
For any finite set $U$ and any unlabelled $\cF$-object $\tilde{F} \in \tilde{\cF}[U]$ there are precisely $|U|!$ many symmetries $(F, \sigma) \in \Sym(\cF)[U]$ such that $F$ belongs to the orbit $\tilde{F}$. Consequently:
\[
	\tilde{\cF}^\omega(z) = Z_{\cF^\omega}(z, z^2, z^3, \ldots).
\]
\end{lemma}

\subsection{Constructions on species}

There are many ways to form species of structures by combining other species. Most prominently, composite structures are formed by partitioning a set and endowing both the partition classes and the collection of all classes with additional weighted structures. Derived structures are regular structures over a set of labels together with a distinguished $*$-placeholder that does not count as regular atom. We recall the details following classical literature by Joyal \cite{MR633783} and Bergeron, Labelle and Leroux \cite{MR1629341}.

\subsubsection{Composite structures}

 Let $\cF^\omega$ and $\cG^\nu$ be combinatorial species with non-negative weights. We assume that $\cG^\nu[\emptyset] = \emptyset$. The {\em composition} $\cF^\omega \circ \cG^\nu = (\cF \circ \cG)^\mu$ is a weighted species that describes partitions of finite sets, where each partition class is endowed with a $\cG$-structure, and the collection of partition classes carries an $\cF$-structure. That is, for each finite set $U$
\[
	(\cF \circ \cG)[U] = \bigcup_{\pi}  \cF[\pi] \times \prod_{Q \in \pi} \cG[Q]
\]
with the index $\pi$ ranging over all unordered partitions of $U$ with non-empty partition classes. In other words, $\pi$ is a set of non-empty subsets of $U$ such that $U = \bigcup_{Q \in \pi} Q$ and  $Q \cap Q' = \emptyset$ for all $Q, Q' \in \pi$ with $Q \ne Q'$. The weight of a composite structure $(F, (G_Q)_{Q \in \pi})$ is given by 
\[
	\mu(F, (G_Q)_{Q \in \pi}) = \omega(F) \prod_{Q \in \pi} \nu(G_Q).
\]
For any bijection $\sigma: U \to V$, the corresponding transport function 
\[
(\cF \circ \cG)[\sigma]: (\cF \circ \cG)[U] \to (\cF \circ \cG)[V]
\]
is given as follows. For each element $(F, (G_Q)_{Q \in \pi}) \in (\cF \circ \cG)[U]$ we let $\bar{\pi} = \{ \sigma(Q) \mid Q \in \pi\}$ denote the corresponding partition of $V$ and set
\[
\bar{\sigma}: \pi \to \bar{\pi}, Q \mapsto \pi(Q).
\]
For each $Q \in \pi$ we let 
\[
\sigma|_Q: Q \to \sigma(Q), x \mapsto \sigma(x) \]
 denote the restriction of $\sigma$ to the class $Q$. We set
\[
(\cF \circ \cG)[\sigma](F, (G_Q)_{Q \in \pi}) = (\cF[\bar{\sigma}](F), ( \cG[ \sigma|_Q](G_{\sigma^{-1}(P)}))_{ P \in \bar{\pi}}).
\]
The cycle index sum of the composition is given by 
\[
	Z_{\cF^\omega \circ \cG^\nu}(z_1, z_2, \ldots) = Z_{\cF^\omega}(Z_{\cG^\nu}(z_1, z_2, \ldots), Z_{\cG^{\nu^2}}(z_2, z_4, \ldots), Z_{\cG^{\nu^3}}(z_3, z_6, \ldots), \ldots ).
\]
Here we let $\nu^i$ denote the weighting that assigns to each $\cG$-object $G$ the weight $\nu(G)^i$.

For example, the species $\Set$ given by $\Set[U] = \{U\}$ for all $U$ has cycle index sum given by
\[
Z_{\Set}(z_1, z_2, \ldots) = \exp \left (\sum_{i=1}^\infty \frac{z_i}{i} \right ).
\]
So, given a weighted species $\cG^\nu$, the generating series for multisets of unlabelled $\cG$-objects is given by
\[
\exp\left(\sum_{i=1}^\infty \frac{\tilde{\cG}^{\nu^i}(z^i)}{i} \right ).
\]

\subsubsection{Derived structures}

Let $\cF^\omega$ be a weighted species. The {\em derived} species $(\cF')^\omega$ is defined as follows. For each set $U$ we let $*_U$ denote a placeholder object not contained in $U$. For example, we could set $*_U := U$, as no set is allowed to be an element of itself. By abuse of notation, we will usually drop the index and just refer to it as the $*$-placeholder atom. We set
\[
	\cF'[U] = \cF[U \cup \{ *_U\}].
\]

The weight of an element $F' \in \cF'[U]$ is its $\omega$-weight as an $\cF$-structure. Any bijection $\sigma: U \to V$ may canonically be extended to a bijection
\[
	\sigma': U \cup \{ *_U\} \to V \cup \{ *_V\},
\]
and we set
\[
	\cF'[\sigma] = \cF[\sigma'].
\]
Thus, an $\cF'$-object with size $n$ is an $\cF$-object with size $n+1$, since we do not count the $*$-placeholder. The cycle index sum of $(\cF')^\omega$ is given by the formal derivative
\[
	Z_{(\cF')^\omega}(z_1, z_2, \ldots) = \frac{\text{d}}{\text{d}{z_1}} Z_{\cF^\omega}(z_1, z_2, \ldots).
\]

\subsection{P\'olya-Boltzmann distributions for composite structures}
\label{sec:boltzmann}
Given a weighted species $\cF^\omega$ and a parameter $y>0$ with $0 < \tilde{\cF}^\omega(y) < \infty$, we may consider the corresponding Boltzmann probability measure
\[
	\mathbb{P}_{\tilde{\cF}^\omega, y}(\tilde{F}) = \tilde{\cF}^\omega(y)^{-1} y^{|\tilde{F}|}\omega(\tilde{F}), \qquad \tilde{F} \in \mathscr{U}(\cF).
\]
Likewise, given parameters $y_1, y_2, \ldots \ge 0$ with 
\[
	0 < Z_{\cF^\omega}(y_1, y_2, \ldots) < \infty,
\]
we may consider the P\'olya-Boltzmann distribution
\[
	\mathbb{P}_{Z_{\cF^\omega}, (y_j)_j}(F, \sigma) = Z_{\cF^\omega}(y_1, y_2, \ldots)^{-1} \frac{\omega(F)}{k!} y_1^{\sigma_1} \cdots y_k^{\sigma_k}
\]
for
\[
 (F, \sigma) \in \bigcup_{k \ge 0} \Sym(\cF)[k].
\]
Note that if we condition a $\mathbb{P}_{\tilde{\cF}^\omega, y}$-distributed random variable on having a fixed size $n$, then the result gets drawn from $\mathscr{U}_n(\cF)$ with probability proportional to its $\omega$-weight. In a way, this is analogous to the fact that simply generated trees (with analytic weights) may be viewed as Galton--Watson trees conditioned on having a fixed number of vertices, and the viewpoint is equally useful in this context.

Lemma~\ref{le:relation} implies the useful fact, that the orbit of the $\cF$-object of a $\mathbb{P}_{Z_{\cF^\omega}, (y, y^2, \ldots)}$-distributed symmetry follows a $\mathbb{P}_{\tilde{\cF}^\omega, y}$-distribution. This provides a systematic way for sampling Boltzmann distributed structures, as the cycle index sums for constructions on species admit explicit expressions with concrete combinatorial interpretations. In particular for composite structures, the following result is given in Bodirsky, Fusy, Kang and Vigerske {\cite[{Prop. 25}]{MR2810913}} for species without weights, and the generalization to the weighted setting is straight-forward.

\begin{lemma}
	\label{le:composition}
	\label{le:sampler}
	Let $\cF^\omega$ and $\cG^\nu$ be weighted species with $\cG[\emptyset] = \emptyset$. Let $y>0$ be a parameter with \[
	\widetilde{\cF^\omega \circ \cG^\nu}(y) =  Z_{\cF^\omega}(\cG^{\nu}(y), \cG^{\nu^2}(y^2), \cG^{\nu^3}(y^3), \ldots) \in ]0, \infty[. \]
	Then the following procedure terminates with an unlabelled $(\cF^\omega \circ \cG^\nu)$-object that follows a $\mathbb{P}_{\widetilde{\cF^\omega \circ \cG^\nu}, y}$-distribution.
	\begin{enumerate}
		\item Let $(\mF, \sigma)$ be a $\mathbb{P}_{Z_{\cF^\omega}, (\tilde{\cG}^\nu(y),\tilde{\cG}^{\nu^2}(y^2), \ldots) }$-distributed $\cF$-symmetry.
		\item For each cycle $\tau$ of $\sigma$ let $|\tau|$ denote its length and draw a $\cG$-object $G_\tau$ according to a $\mathbb{P}_{\tilde{\cG}^{\nu^{|\tau|}}, y^{|\tau|}}$-distribution.
		\item Construct an $\cF \circ \cG$-object by assigning for each cycle $\tau$  and each atom $v$ of $\tau$ an identical copy of $G_\tau$ to $v$.
	\end{enumerate}
\end{lemma}

\subsection{Subexponential sequences}

Subexponential sequences correspond up to tilting and rescaling to subexponential densities of random variables with values in a lattice, and may be put into the more general context of subexponential distributions \cite{MR0348393, MR714482, MR3097424}.

\begin{definition}
	Let $d \ge 1$ be an integer. A power series $g(z) = \sum_{n =0}^\infty g_n z^n$ with non-negative coefficients and radius of convergence $\rho >0$ belongs to the class $\mathscr{S}_d$, if $g_n=0$ whenever $n$ is not divisible by $d$, and
	\begin{align}
	\label{eq:condition}
	\frac{g_n}{g_{n+d}} \sim \rho^d, \qquad \frac{1}{g_n}\sum_{i+j=n}g_ig_j \sim 2 g(\rho) < \infty
	\end{align}
	as $n \equiv 0 \mod d$ becomes large.
\end{definition}

We are going to make use of the following basic properties of subexponential sequences.

\begin{lemma}[{\cite[Theorems 4.8, 4.11, 4.30]{MR3097424}}, \cite{MR772907}]
	\label{le:subexp}
	Let $g(z)$ belong to $\mathscr{S}_d$ with radius of convergence $\rho$.
	\begin{enumerate}
		\item For each $\epsilon>0$ there is an $n_0>0$ such that for all $n \ge n_0$ with $n \equiv 0 \mod d$ and each $k \ge 0$
		\[
		[z^n] g(z)^k \le c(\epsilon) (g(\rho) + \epsilon)^k [z^n] g(z).
		\]
		\item If $f(z)$ is a non-constant power series with non-negative coefficients that is analytic at $\rho$, then $f(g(z))$ belongs to $\mathscr{S}_d$ and
		\[
		[z^n] f(g(z)) \sim f'(g(\rho)) [z^n]g(z), \qquad n \to \infty, \qquad n \equiv 0 \mod d.
		\]
		\item If $a_n = h(n) n^{-\beta} \rho^{-n}$ for some constants $\rho>0$, $ \beta > 1$ and a slowly varying function $h$, then the series $\sum_{n \in d\ndN} a_n z^n$ belongs to the class $\mathscr{S}_d$.
	\end{enumerate}
\end{lemma}

The following criterion will prove to be useful as well.

\begin{lemma}[{\cite[Thm. 4.9]{MR3097424}}]
	\label{le:help}
	Let $f(z)$ belong to $\mathscr{S}_1$ with radius of convergence $\rho$, and $g_1(z), g_2(z)$ be power-series with non-negative coefficients. If
	\[
		\frac{[z^n]g_1(z)}{[z^n]f(z)} \to c_1 \qquad \text{and} \qquad  \frac{[z^n]g_2(z)}{[z^n]f(z)} \to c_2
	\]
	as $n \to \infty$ with $c_1,c_2 \ge 0$, then
	\[
		 \frac{[z^n]g_1(z)g_2(z)}{[z^n]f(z)} \to c_1g_2(\rho) + c_2g_1(\rho).
	\]
	If additionally $c_1g_2(\rho) + c_2g_1(\rho)>0$, then $g_1(z)g_2(z)$ belongs to $\mathscr{S}_1$.
\end{lemma}

\section{Unlabelled Gibbs partitions}
\label{sec:conv}
Let $\cF^\omega$ and $\cG^\nu$ be weighted combinatorial species with $\cG[\emptyset] = \emptyset$, so that the weighted composition
\[
	(\cF \circ \cG)^\mu = \cF^\omega \circ \cG^\nu
\]
is well-defined. Throughout we assume that $[z^k]\tilde{\cF}^\omega(z) >0$ for at least one $k \ge 1$ and that $\widetilde{\cF^\omega \circ \cG^\nu}(z)$ is not a polynomial. For each integer $n \ge 0$ with \[[z^n] \widetilde{\cF^\omega \circ \cG^\nu}(z) > 0\] we may sample a random composite structure 
\[
\mS_n = (\mF_n, (\mG_Q)_{Q \in \pi_n})
\] from the set $\mathscr{U}_n(\cF \circ \cG)$ with probability proportional to its $\mu$-weight. 

We are going to study the asymptotic behaviour of the remainder $\mR_n$ when deleting "the" largest component from $\mS_n$.  More specifically, we pick an arbitrary representative of $\mS_n$ and construct $\mR_n$ as follows. We make a uniform choice of a component $Q_0 \in \pi_n$ having maximal size, and let $\mF_n'$ denote the $\cF'$-object obtained from the $\cF$-object $\mF_n$ by relabeling the $Q_0$ atom of $\mF_n$ to a $*$-placeholder.

Thus \[\mF_n' = \cF[\gamma](\mF_n) \in \cF'[\pi_n \setminus \{Q_0\}]\] for the  bijection $\gamma: \pi_n \to (\pi_n \setminus \{Q_0\}) \cup \{*\}$ with $\gamma(Q_0)=*$ and $\gamma(Q) = Q$ for $Q \ne Q_0$. This yields an unlabelled $\cF' \circ \cG$-object
\[
\mR_n := (\mF_n', (\mG_Q)_{Q \in \pi_n \setminus \{Q_0\}}) \in \mathscr{U}(\cF' \circ \cG).
\]

We let $\rho$ denote the radius of convergence of the ordinary generating series  $\tilde{\cG}^\nu(z)$ and suppose throughout that 
	\begin{align}
	\label{eq:a1}
	Z_{\cF^\omega}( \tilde{\cG}^\nu(\rho) + \epsilon, \tilde{\cG}^{\nu^2}((\rho + \epsilon)^2), \tilde{\cG}^{\nu^3}((\rho + \epsilon)^3), \ldots) < \infty
	\end{align}
	for some $\epsilon >0 $. Let $\mR$ be a random unlabelled $\cF' \circ \cG$-element  that follows a Boltzmann distribution
	\[
	\Pr{\mR = R} = \frac{\mu(R) \rho^{|R|}}{ \widetilde{(\cF')^\omega \circ  \cG^\nu}(\rho)}, \qquad R \in \mathscr{U}(\cF' \circ \cG).
	\]

\begin{theorem}
	\label{te:main}
	If the series $\tilde{\cG}^\nu(z)$ belongs to the class $\mathscr{S}_d$, then
	 \begin{align*}
	 \label{eq:convergence}
	 d_{\textsc{TV}}(\mR_n, \mR) \to 0,  \qquad n\to \infty, \qquad   n \equiv 0 \mod d.
	 \end{align*}
\end{theorem}
The main challenge for verifying Theorem~\ref{te:main} is that we consider objects up to symmetry. Lemma~\ref{le:sampler} provides a way of sampling $\mS_n$ as a conditioned Boltzmann-distributed composite structure consisting of an $\cF$-symmetry with identical $\cG$-objects dangling from each cycle. The key idea will be that the largest $\cG$-object is likely to correspond to a fixpoint of the symmetry. A similar congelation phenomenon was observed for random labelled composite structures sampled from $(\cF^\omega \circ \cG^\nu)[n]$ with probability proportional to their weight~\cite[Thm.3.1]{Mreplaceme}. Our overall strategy is similar, but treating unlabelled structures is more involved. We require the following enumerative result for the proof of our main theorem. 
\begin{lemma}
	\label{le:enumerative}
	If the series $\tilde{\cG}^\nu(z)$ belongs to the class $\mathscr{S}_d$, then 
	\[
		[z^n] \widetilde{\cF^\omega \circ \cG^\nu}(z) \sim \widetilde{(\cF')^\omega \circ \cG^\nu}(\rho) [z^n] \tilde{\cG}^\nu(z), \qquad n \to \infty, \qquad n \equiv 0 \mod d
	\]
	with
	\[
	\widetilde{(\cF')^\omega \circ \cG^\nu}(\rho) = \left( \frac{\text{d}}{\text{d}{z_1}} Z_{\cF^\omega} \right) (\tilde{\cG}^\nu(\rho), \tilde{\cG}^{\nu^2}(\rho^2), \ldots).
	\]
\end{lemma}
If $\tilde{\cG}^\nu(z)$ is amendable to singularity analysis, then Lemma~\ref{le:enumerative} may also be verified using analytic methods \cite{MR2483235}. But we make no assumptions at all about the singularities of $\tilde{\cG}^\nu(z)$ on the circle $|z| = \rho$. We only require that this series belongs to the class $\mathscr{S}_d$, which is much more general.

Clearly Theorem~\ref{te:main} also implies total variational convergence of the number of components, which has been studied in \cite{MR1763972} for the case of weighted multisets where  $\cF^\omega = \Set$ and each $\cF$-object receives weight $1$.

\begin{corollary}
	\label{co:moments}
Suppose that the series $\tilde{\cG}^\nu(z)$ belongs to the class $\mathscr{S}_d$. Let $c(\cdot)$ denote the number of components in a composite structure. Then $c(\mS_n)$ converges towards $1 + c(\mR)$ in total variation.
\end{corollary}

If we condition $\mR_n$ on having a fixed size $k < n/2$, then the $\cG$-object of the largest component gets drawn with probability proportional to its weight from $\mathscr{U}_{n - k}$. And clearly $\mR$ has with probability tending to $1$ size less than $n/2$. Hence we may rephrase Theorem~\ref{te:main} as follows.

\begin{corollary}
 Suppose that the series $\tilde{\cG}^\nu(z)$ belongs to the class $\mathscr{S}_d$.
 If $\mR$ has size less than $n$, let $\hat{\mS}_n$ denote the random unlabelled $\cF\circ \cG$-object constructed by drawing a $\cG$-object $\mG_{n - |\mR|}$ from $\mathscr{U}_{n - |\mR|}$ with probability proportional to its weight. If $\mR \ge n$, set $\hat{\mS}_n$ to some placeholder value $\diamond$. Then
\[
	d_{\textsc{TV}}(\mS_n, \hat{\mS}_n) \to 0, \qquad n \to \infty, \qquad n \equiv 0 \mod d.
\]
\end{corollary}

\section{Proofs}
\label{sec:proofs}

Before starting with the proofs of our main results, we make an  elementary observation.

\begin{lemma}
	\label{le:cyclepgf}
	Let $\cF^\omega$ and $\cG^\nu$ be weighted species with $\cG^\nu[\emptyset] = \emptyset$, and let $(\mS, \sigma)$ be a random symmetry that  follows a $\mathbb{P}_{Z_{\cF^\omega \circ \cG^\nu}, (\rho^j)_j}$-distribution for some $\rho>0$. The composite structure of $\mS$ is of the form $(\mF, (\mG_Q)_{Q \in \pi})$ with $\pi$ a partition of a finite set, $\mF$ an $\cF$-structure on $\pi$, and $G_Q$ a $\cG$-structure on $Q$ for each $Q \in \pi$. As $\sigma$ is an automorphism, it follows that
	\[
	\bar{\sigma}: \pi \to \pi, Q \mapsto \sigma(Q)
	\]
	is well-defined permutation of the collection $\pi$ of partition classes. For each $i \ge 1$ let $X_i$ denote the number of cycles of length $i$ in in the induced permutation $\bar{\sigma}$, $Y_i = \sum_i X_i$ the total number of atoms contained in cycles of length $i$, and $Z_i$ the sum of sizes of all $\cG$-objects corresponding to atoms of $\bar{\sigma}$ that are contained in cycles of length $i$. Then
	\[
		\Exb{ \prod_{i \ge 1} x_i^{X_i}y_i^{Y_i}z_i^{Z_i}} = \frac{Z_{\cF^\omega}(x_1 y_1 \tilde{\cG}^\nu(z_1 \rho), x_2 y_2^2 \tilde{\cG}^{\nu^2}( (z_2 \rho)^2), x_3 y_3^3 \tilde{\cG}^{\nu^3}( (z_3 \rho)^3), \ldots ) }{\widetilde{\cF^\omega \circ \cG^\nu}(\rho)}.
	\]	
\end{lemma}
Lemma~\ref{le:cyclepgf} is a minor extension of the proof of the well-known enumerative formula
\[
	\widetilde{\cF^\omega \circ \cG^\nu}(z) = Z_{\cF^\omega}(\tilde{\cG}^\nu(z), \tilde{\cG}^{\nu^2}(z^2), \ldots)
\]
given for example in \cite[Theorem 3 and Section 6]{MR633783} or \cite[Proposition 11 of Section 2.3]{MR1629341}. Instead of using a single formal variable $z$ in the proof for counting the total size, all involved counting series may be replaced by versions with additional formal variables $(x_i, y_i, z_i)_{i \ge 1}$, that keep track of the required fine-grained statistics. We do not aim to go through the details. Roughly speaking, the idea behind this is that symmetries of composite $\cF \circ \cG$-structures correspond, up to a certain relabelling and cycle composition process, to an $\cF$-symmetry, where each cycle $\tau$ with length $|\tau|$ gets endowed with $|\tau|$ identical copies of a $\cG$-symmetry. Thus, in the sum
\[
Z_{\cF^\omega}(x_1 y_1 \tilde{\cG}^\nu(z_1 z), x_2 y_2^2 \tilde{\cG}^{\nu^2}( (z_2 z)^2), x_3 y_3^3 \tilde{\cG}^{\nu^3}( (z_3 z)^3), \ldots ),
\]
the variable $z$ keeps track of the total size, the $x_i$ of the number of cycles of length $i$ in the symmetry and consequently the $y_i$ of the total mass of theses cycles. The powers $(z_i z)^i$ are due to the fact each $\cG$-object assigned to a cycle with length $i$ gets counted $i$-times due to the identical copies corresponding to each atom of the cycle.

\begin{proof}[Proof of Lemma~\ref{le:enumerative}]
Throughout, we let $n$ denote an integer that is divisible by $d$.
We assumed that $\cF^\omega$ and $\cG^\nu$ are weighted species such that the ordinary generating function $\tilde{\cG}^\nu(z)$ belongs to $\mathscr{S}_d$. We further assumed by Inequality~\eqref{eq:a1} that
\begin{align}
	\label{eq:assumption}
		Z_{\cF^\omega}( \tilde{\cG}^\nu(\rho) + \epsilon, \tilde{\cG}^{\nu^2}((\rho + \epsilon)^2), \tilde{\cG}^{\nu^3}((\rho + \epsilon)^3), \ldots) < \infty
\end{align}
for some $\epsilon >0$, with $\rho$ denoting the radius of convergence of the series $\tilde{\cG}^\nu(z)$.

We start by constructing a $\mathbb{P}_{\widetilde{\cF^\omega \circ \cG^\nu}, \rho}$-distributed composite structure according to Lemma~\ref{le:sampler}. Let $(\mF, \sigma)$ follow a  $\mathbb{P}_{Z_{\cF^\omega}, (\tilde{\cG}^\nu(\rho),\tilde{\cG}^{\nu^2}(\rho^2), \ldots) }$-distribution. For each cycle $\tau$ of $\sigma$ let $|\tau|$ denote its length and draw a $\cG$-object $G_\tau$ according to a $\mathbb{P}_{\tilde{\cG}^{\nu^{|\tau|}}, \rho^{|\tau|}}$-distribution. We construct the $\cF \circ \cG$-object $\mS$ by assigning for each cycle $\tau$  and each atom $v$ of $\tau$ an identical copy of $G_\tau$ to $v$. Thus $\mS$ corresponds to $(\mF, (\mG_v)_v)$.

Let $f$ denote the number of fixpoints of the permutation $\sigma$, and $\mG_1, \ldots \mG_f$ the corresponding $\cG$-structures. We set $g_i = |\mG_i|$ for all $i$. Let $\mH$ denote the structure obtained from $\mS$ by deleting all $\cG$-objects that correspond to fixpoints of $\sigma$, and let $h$ denote the total size of its remaining $\cG$-objects. Thus
\begin{align}
\label{eq:ff3}
|\mS| = \sum_{i=1}^f g_i + h.
\end{align}
The $(g_i)_i$ are independent, but $f$ and $h$ may very well depend on each other. By Lemma~\ref{le:cyclepgf}, their joint probability generating function is given by
\begin{align}
	\label{eq:gim1}
	\Ex{y^f w^h} = \frac{Z_{\cF^\omega}(y \tilde{\cG}^\nu(\rho),\tilde{\cG}^{\nu^2}(w^2\rho^2), \tilde{\cG}^{\nu^3}(w^3\rho^3), \ldots)}{\widetilde{\cF^\omega \circ \cG^\nu}(\rho)}.
\end{align}
Hence the assumption \eqref{eq:assumption} states precisely that the vector $(f,h)$ has finite exponential moments. We are going to show that
\begin{align}
	\label{eq:hh}
	\Pr{|\mS|=n} \sim \Ex{f} \Pr{g = n}
\end{align}
$g$ denoting the size of a $\mathbb{P}_{\tilde{\cG}^\nu, \rho}$-distributed random $\cG$-object. Since Equation\eqref{eq:gim1} implies that
\[
	\Ex{f} = \frac{\frac{\text{d}}{\text{d}{z_1}} Z_{\cF^\omega}( \tilde{\cG}^\nu(\rho),\tilde{\cG}^{\nu^2}(\rho^2), \tilde{\cG}^{\nu^3}(\rho^3), \ldots)\tilde{\cG}^\nu(\rho)}{\widetilde{\cF^\omega \circ \cG^\nu}(\rho)},
\]
it is clear that Equation~\eqref{eq:hh} is equivalent to 
 	\[
 	[z^n] \widetilde{\cF^\omega \circ \cG^\nu}(z) \sim \widetilde{(\cF')^\omega \circ \cG^\nu}(\rho) [z^n] \tilde{\cG}^\nu(z), \qquad n \to \infty, \qquad n \equiv 0 \mod d.
 	\]
We have thus successfully reduced the task of determining asymptotically the coefficients of $\widetilde{\cF^\omega \circ \cG^\nu}(z)$ to the probabilistic task of verifying \eqref{eq:hh}, and may apply available results for subexponential probability distributions.  Equation~\eqref{eq:ff3} implies that 
\begin{align}
	\label{eq:tmp}
	\Pr{|\mS|=n} = \Prb{\sum_{i=1}^f g_i + h=n} 
		= \sum_{k\ge 0} \Pr{f=k} \Prb{\sum_{i=1}^k g_i + h=n \mid f = k}.
\end{align}
Let $g$ denote a random variable that is distributed like the size of a $\mathbb{P}_{\tilde{\cG}^\nu, y}$-distributed random $\cG$-object.  
Given $f=k$, the $(g_i)_{1 \le i \le k}$ are independent and identically distributed copies of $g$.
Lemma~\ref{le:subexp} implies that for each fixed $k$ it holds that
\[
	\Prb{\sum_{i=1}^k g_i = n \mid f = k} = \Prb{\sum_{i=1}^k g_i =n} \sim k \Pr{g = n}.
\]
As the vector $(f,h)$ has finite exponential moments, it also holds that the conditioned version $(h \mid f = k)$ has finite exponential moments.   It follows from  Lemma~\ref{le:help} that
\[
\Prb{\sum_{i=1}^k g_i + h=n \mid f = k} \sim  k \Pr{g= n}
\]
and hence
\[
\Prb{\sum_{i=1}^k g_i + h=n , f = k} \sim  \Pr{f=k} k \Pr{g= n}.
\]
Consequently, if we can find a summable sequence $(C_k)_{k \ge 0}$ such that
\begin{align}
\label{eq:bound}
\Prb{\sum_{i=1}^k g_i + h=n , f = k} \le C_k \Pr{g=n},
\end{align}
for all $k$, then it follows by dominated convergence  that
\[
\Pr{|\mS|=n} = \sum_{k\ge 0} \Pr{f=k} \Prb{\sum_{i=1}^k g_i + h=n \mid f = k} \sim \Ex{f} \Pr{g=n}.
\]
Thus, in order to show \eqref{eq:hh} it remains to establish Inequality~\eqref{eq:bound}. 
By Lemma~\ref{le:subexp} for each $\epsilon>0$ there is an integer $x_0 = x_0(\epsilon)>0$ and a constant $c(\epsilon)>0$ such that for all integers $x \ge x_0$ and each $k \ge 0$ it holds that
\begin{align}
\label{eq:help}
\Prb{\sum_{i=1}^k g_i = x} \le c(\epsilon)(1+\epsilon)^k \Pr{g = x}.
\end{align}
Clearly we have that
\begin{multline}
	\label{eq:ff7}
	\Prb{\sum_{i=1}^k g_i + h=n, f=k} =  \Prb{\sum_{i=1}^k g_i + h=n, h > n- x_0, f=k} \\ + \Prb{\sum_{i=1}^k g_i + h=n, h \le n- x_0,f=k}.
\end{multline}
Since $h$ has finite exponential moments, there are constants $C,c>0$ such that for all $n$
\[
\Prb{\sum_{i=1}^k g_i + h=n, h > n- x_0, f=k} \le \Pr{h > n - x_0} \le C \exp(-cn).
\]
We know that $g$ is heavy-tailed because it belongs to $\mathscr{S}_d$. Hence it follows that
\begin{align}
\label{eq:firstsummand}
\Prb{\sum_{i=1}^k g_i + h=n, h > n- x_0, f=k} = o(\Pr{g=n})
\end{align}
uniformly for all $k \ge 0$ as $n$ becomes large.
As for the other summand in \eqref{eq:ff7}, it holds that
\begin{multline*}
	\Prb{\sum_{i=1}^k g_i + h=n, h \le n- x_0, f=k}
					\\= \sum_{\ell =0}^{n-x_0} \Pr{h=\ell,f=k} \Prb{\sum_{i=1}^k g_i = n-\ell  }.
\end{multline*}
Since $n-k \ge x_0$, it follows from Inequality~\eqref{eq:help} that for all $\epsilon>0$
\begin{multline}
	\label{eq:last}
	\sum_{\ell =0}^{n-x_0} \Pr{h=\ell,f=k} \Prb{\sum_{i=1}^k g_i = n-\ell  } \\\le c(\epsilon)  \sum_{\ell=0}^{n} \Pr{h=\ell,f=k} (1+ \epsilon)^k \Pr{g =n-\ell}.
\end{multline}
As the vector $(f,h)$ has finite exponential moments, there is a $\delta>0$ such that 
\begin{align}
\label{eq:finite}
\Ex{(1+\delta)^f(1+\delta)^w} < \infty.
\end{align}
Since $\epsilon>0$ was arbitrary, we may choose it small enough such that $0 < \epsilon < \delta$. Thus
\[
\frac{1+\epsilon}{1 + \delta} < 1
\]
and
\begin{align}
\label{eq:second}
&\Prb{\sum_{i=1}^k g_i + h=n, h \le n- x_0, f=k}  \nonumber \\ 
&\qquad\le c(\epsilon) \left( \frac{1+\epsilon}{1 + \delta}\right )^k \sum_{\ell=0}^{n} \Pr{h=\ell,f=k} (1+ \delta)^k \Pr{g =n-\ell} \nonumber \\
&\qquad \le c(\epsilon) \left( \frac{1+\epsilon}{1 + \delta}\right )^k \sum_{\ell = 0}^n p_\ell \Pr{g = n-\ell} 
\end{align}
with
\[
	p_\ell = \sum_{k \ge 0} \Pr{h=\ell, f=k}(1+ \delta)^k
\]
satisfying
\[
	\sum_{\ell \ge 0} p_\ell (1 + \delta)^\ell < \infty
\]
by Inequality~\eqref{eq:finite}. Hence we may apply Lemma~\ref{le:help} to obtain
\[
\sum_{\ell = 0}^n p_\ell \Pr{g = n-\ell} \sim \Pr{g = n}.
\]
So Equation~\eqref{eq:firstsummand} and Inequality~\eqref{eq:second} imply that for all $k \ge 0$
\[
\Prb{\sum_{i=1}^k g_i + h=n, f=k} \le C_k \Pr{g=n}
\]
for a summable sequence $(C_k)_{k \ge 0}$. This verifies Inequality~\eqref{eq:bound} and hence \eqref{eq:hh} follows by dominated convergence.
\end{proof}

\begin{proof}[Proof of Theorem \ref{te:main}]
	We use the same notation as in the proof of Lemma~\ref{le:enumerative}, that is,  we let $\mS$ denote a random $\mathbb{P}_{\widetilde{\cF^\omega \circ \cG^\nu}, \rho}$-distributed composite structure assembled according to Lemma~\ref{le:sampler} as follows: We sample an $\cF$-symmetry $(\mF, \sigma)$ following a $\mathbb{P}_{Z_{\cF^\omega}, (\tilde{\cG}^\nu(\rho),\tilde{\cG}^{\nu^2}(\rho^2), \ldots) }$-distribution and let $f$ denote the number of fixpoints of $\sigma$. We let $(\mG_i)_{i \ge 1}$ denote an independent family of $\mathbb{P}_{\tilde{\cG}^\nu, \rho}$-distributed $\cG$-objects, of which we match the first $f$ to the fixpoints of $\sigma$ in any canonical order. For example, we may order the fixpoints according to their labels in $\{1, \ldots, |\mF|\}$, but any canonical order will do by exchangability of the $\mG_i$.  Likewise, for each cycle $\tau$ of $\sigma$ with length $|\tau| \ge 2$  we draw a $\cG$-object $G_\tau$ according to a $\mathbb{P}_{\tilde{\cG}^{\nu^{|\tau|}}, \rho^{|\tau|}}$-distribution, and assign to each atom of $\tau$ an identical copy of $G_\tau$. We let $\mH$ denote structure obtained from $(\mF, \sigma)$ by only attaching the $\cG$-objects to atoms of cycles with length at least $2$. Then $\mS$ is fully described by the vector
	\[
		(\mH, \mG_1, \ldots, \mG_f).
	\]
	It holds that
	\[
		|\mS| = \sum_{i=1}^f g_i + h.
	\]
	with $h$ denoting the number of atoms of $\mH$ and $g_i = |\mG_i|$ for all $i$.
	As discussed in Subsection~\ref{sec:boltzmann}, the result of conditioning a Boltzmann-distributed object on having a fixed size gets sampled with probability proportional to its weight among all objects of this size. Hence
	\begin{align*}
		\mS_n \eqdist (\mS \mid |\mS| = n).
	\end{align*}

	Similarly, the $\mathbb{P}_{\widetilde{(\cF')^\omega \circ \cG^\nu}, \rho}$-distributed $\cF' \circ \cG$-object $\mR$ may, by virtue of Lemma~\ref{le:sampler}, be sampled analogously as follows: We draw an $\cF$-symmetry $(\mF', \sigma)$ following a $\mathbb{P}_{Z_{(\cF')^\omega}, (\tilde{\cG}^\nu(\rho),\tilde{\cG}^{\nu^2}(\rho^2), \ldots) }$-distribution and let $f'$ denote the number of fixpoints of $\sigma'$. Note that $\sigma$ is a permutation of the non-$*$-atoms of $\mF$, hence we do not count the place-holder atom. We let $(\mG^i)_{i \ge 1}$ denote a list of independent copies of a $\mathbb{P}_{\tilde{\cG}^\nu, \rho}$-distributed $\cG$-object, and match the first $f'$ to the fixpoints of $\sigma'$ in a canonical way. For each cycle $\tau$ of $\sigma'$ with length $|\tau| \ge 2$  we draw a $\cG$-object $G'_\tau$ according to a $\mathbb{P}_{\tilde{\cG}^{\nu^{|\tau|}}, \rho^{|\tau|}}$-distribution, and assign to each atom of $\tau$ an identical copy of $G_\tau$. We let $\mH'$ denote the pruned structure where only atoms of non-fixpoints of $\sigma'$ receive a $\cG$-object. Thus $\mR$ is fully determined by the vector
	\[
		(\mH', \mG^1, \ldots, \mG^{f'}),
	\]
	and we set $g^i = |\mG^i|$ for all $i$ and let $h'$ denote the number of atoms in $\mH'$. If $\mR$ has size less than $n$, we let $\hat{\mS}_n$ denote the result of assigning to the $*$-placeholder atom a random unlabelled $\cG$-structure $\mG^*$ sampled from $\mathscr{U}_{n - |\mR|}(\cG)$ with probability proportional to its weight. If $\mR \ge n$, we let $\hat{\mS}_n$ assume some placeholder value $\hat{\mS}_n = \diamond$. We are going to show that
	\begin{align}
		\label{eq:showme}
		d_{\textsc{TV}}(\mS_n, \hat{\mS}_n) \to 0, \qquad n \to \infty, \qquad n \equiv 0 \mod d.
	\end{align}
	If $\mR < n/2$, then $\mG^*$ is the largest $\cG$-object of $\hat{\mS}_n$. Since $\mR$ is almost surely finite, this event takes place with probability tending to $1$ as $n$ becomes large. Hence \eqref{eq:showme} implies that
	\begin{align*}
		d_{\textsc{TV}}(\mR_n, \mR) \to 0.
	\end{align*}
	Thus verifying \eqref{eq:showme} is sufficient to conclude the proof. 
	
	If we interpret $\mF'$ as an $\cF$-object $\mF'_*$ (rather than an $\cF'$-object), then the permutation $\sigma'$ extends to an $\cF$-automorphism $\sigma'_*$ of $\mF'_*$ such that the $*$-vertex is a fixpoint.  The distributions of $(\mF, \sigma)$ and $(\mF'_*, \sigma'_*)$ differ in the fact that $\sigma'$ always has at least one fixpoint, and that the probability to assume a fixed size is different. However, given integers $m, k \ge 1$  it holds that up to relabelling
	\begin{align}
		\label{eq:invariant}
		((\mF, \sigma) \mid f = k, |\mF| = m) \eqdist ((\mF'_*, \sigma'_*) \mid f' = k-1, |\mF'| = m-1).
	\end{align}
	This may be verified as follows. The left-hand side gets drawn with probability proportional to its weight from the subset $A_k \subset \Sym(\cF)[m]$ of all symmetries with $k \ge 1$ fixpoints. Likewise, the right-hand side gets drawn with probability proportional to its weight from the subset $B_k \subset \Sym(\cF')[m-1]$ of symmetries with $k$ fixpoints in total (counting the $*$-atom).  There is a weight-preserving bijection between $\Sym(\cF')[m-1]$ and the symmetries in $\Sym(\cF)[m]$ where the atom $m$ is a fixpoint. It follows that there is a weight-preserving $1$ to $m$ correspondence between  $\Sym(\cF')[m-1]$ and the set of symmetries in $\Sym(\cF)[m]$ with a distinguished fixpoint. Now, for each symmetry in $A_k$ there are precisely $k$ ways to distinguish a fixpoint, hence there is a weight-preserving $1$ to $km$ relation between $A_k$ and $B_k$. Thus \eqref{eq:invariant} holds.
	
	Let $x_1, \ldots, x_k \ge 1$ and $r \ge 0$ be given with
	\[
	x_1 + \ldots + x_k + r = n.
	\]
	It follows from \eqref{eq:invariant} and the construction of $\mH$ and $\mH'$, that 
	\[
		(\mH \mid f=k, h= r) \eqdist (\mH' \mid f' = k-1, h'= r).
	\]
	If we condition the left-hand side additionally on $g_i = x_i$ for all $1 \le i \le k$, then the distribution of $\mH$ does not change and $\mG_i$ gets drawn from $\tilde{\cG}[x_i]$ with probability proportional to its $\nu$-weight. Likewise, if we condition the right-hand side additionally on $g^i=x_i$ for all $1 \le i \le k-1$, then for each $i$ it holds that $\mG^i$ gets drawn from $\tilde{\cG}[x_i]$ with probability proportional to its weight, and $\mG^*$ gets drawn with probability proportional to its weight among all unlabelled $\cG$-objects with $n - r - x_1 - \ldots -x_{k-1} = x_k$ atoms.  Thus
	\begin{multline}
		\label{eq:established}
		(\mS \mid f=k, h=r, g_1=x_1, \ldots, g_k = x_k) \\ \eqdist (\hat{\mS}_n \mid f'=k-1, h'=r, g^1=x_1, \ldots, g^{k-1} = x_{k-1}).
	\end{multline}
		We let $g$ denote a random variable that is distributed like the size of a random $\cG$-object with a $\mathbb{P}_{\tilde{\cG}^\nu, \rho}$ distribution. Since $\tilde{\cG}^\nu(z)$ belongs to $\mathscr{S}_d$, it holds that
		\[
		\Pr{g = n+d} \sim \Pr{g = n}, \qquad n \to \infty.
		\]
		This implies that there is a sequence $t_n$ of non-negative integers such that $t_n \to \infty$ and
		\begin{align}
		\label{eq:admissible}
		\lim_{n \to \infty} \sup_{\substack{0 \le y \le t_n \\ y  \equiv 0 \mod d}}  | \Pr{g = n + y}/\Pr{g = n} -1| = 0.
		\end{align}
		Without loss of generality we may assume that $t_n < n/2$ for all $n$. 
		For any sequence $\mathbf{y} = (y_1, \ldots, y_{k-1})$ of positive integers  we set \[D(\mathbf{y}) := y_1 + \ldots, y_{k-1}. \]  For each integer $m$ with $m > D(\mathbf{y})$ we also set
		\[
		\sigma_m(\mathbf{y}) := \{ (y_1, \ldots, y_{j-1}, m - D(\mathbf{y}), y_j, \ldots, y_k) \mid 1 \le j \le k\}.
		\]
		Finally, we set
		\[
		M_n := \{ (k, r, \mathbf{y}) \mid k \ge 1, r \ge 0, \mathbf{y} \in \ndN^{k-1}, r + D(\mathbf{y}) \le t_n \}.
		\]
		We will show that as $n$ becomes large, it holds uniformly for all $(k,r, \mathbf{y}) \in M_n$ that
		\begin{multline}
			\label{eq:showmenow}
			\Pr{ f=k, h=r, (g_1, \ldots, g_k) \in \sigma_{n-r}(\mathbf{y}) \mid g_1 + \ldots + g_f + h = n} \\ \sim \Pr{ f'=k-1, h' =r, (g^1, \ldots, g^{k-1}) = \mathbf{y} }.
		\end{multline}
		For $D(\mathbf{y}) + r \le  t_n < n/2$, the $(g_1, \ldots, g_k) \in \sigma_n(\mathbf{y})$ corresponds to $k$ distinct outcomes, depending on the unique location for the maximum of the $g_i$. Thus the  left-hand side in \eqref{eq:showmenow} divided by the right-hand side  equals
		\[
			\frac{k \Pr{f=k,h=r} \Pr{g=n - D(\mathbf{y})-r } }{\Pr{f'=k-1,h'=r} \Pr{g_1 + \ldots + g_f +h = n}}.
		\]
		Note that
		\[
			\frac{k \Pr{f=k,h=r}}{\Pr{f'=k-1,h'=r}} = \frac{\tilde{\cG}^\nu(\rho) \widetilde{(\cF')^\omega \circ \cG^\nu}(\rho)}{\widetilde{\cF^\omega \circ \cG^\nu}(\rho)} = \Ex{f}.
		\]
		By Lemma~\ref{le:enumerative} it holds that
		\[
			\Pr{g_1 + \ldots + g_f +h = n} \sim \Ex{f} \Pr{g = n}.
		\]
		Equation~\eqref{eq:admissible} and $D(\mathbf{y}) + r \le  t_n$ yield that 
		\[
		\Pr{g=n - D(\mathbf{y})-r } \sim \Pr{g=n}
		\]
		uniformly for $(k,r,\mathbf{y}) \in M_n$. This verifies the asymptotic equality in \eqref{eq:showmenow}.
	
		As $t_n \to \infty$, it clearly holds that
		\[
			(f'+1,r',(g^1, \ldots, g^{f'})) \in M_n
		\]
		with probability tending to $1$ as $n$ becomes large. Hence it follows from \eqref{eq:showmenow} that
		\[
			\Pr{ (f,h, (g_1, \ldots, g_k)) \in \{ \{(k,r)\} \times \sigma_{n-r}(\mathbf{y}) \mid (k,r, \mathbf{y}) \in M_n \}} \to 1
		\]
		as $n$ becomes large. Thus, we have that uniformly for all sets $\cE$ of $n$-sized unlabelled $\cF \circ \cG$-objects
		\begin{align*}
		\Pr{\mS_n \in \cE} &= \Pr{ \mS \in \cE \mid g_1 + \ldots g_f +h = n} \\ 
		&= o(1) + \sum_{ (k,r,\mathbf{y}) \in M_n} \frac{\Pr{ \mS \in \cE, (f,h)=(k,r),  (g_1, \ldots, g_k) \in \sigma_n(\mathbf{y})}}{\Pr{ g_1 + \ldots g_f + h = n}}.
		\end{align*}
		The summand for $(k, r, \mathbf{y})$ may be expressed by the product
		\begin{multline*}
			\Pr{ \mS \in \cE \mid (f,h)=(k,r), (g_1, \ldots, g_k) \in \sigma_n(\mathbf{y})} \\\Pr{(f,h)=(k,r), (g_1, \ldots, g_k)\in \sigma_n(\mathbf{y}) \mid g_1 + \ldots g_f +h = n}.
		\end{multline*}
		Equation~\eqref{eq:established} yields that the first factor is equal to
		\[
		\Pr{ \hat{\mS}_n \in \cE \mid f' = k-1, h' = r, (g^1, \ldots, g^{k-1}) = \mathbf{y}}.
		\]
		By Equation~\eqref{eq:showmenow}, the second factor is asymptotically equivalent to
		\[
			\Pr{ f'=k-1, h' =r, (g^1, \ldots, g^{k-1}) = \mathbf{y} }
		\]
		uniformly for all $(k,r,\mathbf{y}) \in M_n$ as $n$ becomes large. Thus
		\begin{align*}
			\Pr{\mS_n \in \cE} &= o(1) + \sum_{(k,r,\mathbf{y}) \in M_n} \Pr{\hat{\mS}_n \in \cE, (f',h', (g^i)_i)=(k-1, r, \mathbf{y})} \\
			&= o(1) + \Pr{\hat{\mS}_n \in \cE}.
		\end{align*}
		This completes the proof.
\end{proof}

\bibliographystyle{siam}
\bibliography{unlgibbs}

\begin{thebibliography}{10}

\bibitem{MR2032426}
{\sc R.~Arratia, A.~D. Barbour, and S.~Tavar{\'e}}, {\em Logarithmic
  combinatorial structures: a probabilistic approach}, EMS Monographs in
  Mathematics, European Mathematical Society (EMS), Z\"urich, 2003.

\bibitem{MR2121024}
{\sc A.~D. Barbour and B.~L. Granovsky}, {\em Random combinatorial structures:
  the convergent case}, J. Combin. Theory Ser. A, 109 (2005), pp.~203--220.

\bibitem{MR1763972}
{\sc J.~P. Bell, E.~A. Bender, P.~J. Cameron, and L.~B. Richmond}, {\em
  Asymptotics for the probability of connectedness and the distribution of
  number of components}, Electron. J. Combin., 7 (2000), pp.~Research Paper 33,
  22 pp. (electronic).

\bibitem{MR1629341}
{\sc F.~Bergeron, G.~Labelle, and P.~Leroux}, {\em Combinatorial species and
  tree-like structures}, vol.~67 of Encyclopedia of Mathematics and its
  Applications, Cambridge University Press, Cambridge, 1998.
\newblock Translated from the 1994 French original by Margaret Readdy, With a
  foreword by Gian-Carlo Rota.

\bibitem{MR2810913}
{\sc M.~Bodirsky, {\'E}.~Fusy, M.~Kang, and S.~Vigerske}, {\em Boltzmann
  samplers, {P}\'olya theory, and cycle pointing}, SIAM J. Comput., 40 (2011),
  pp.~721--769.

\bibitem{MR0348393}
{\sc J.~Chover, P.~Ney, and S.~Wainger}, {\em Functions of probability
  measures}, J. Analyse Math., 26 (1973), pp.~255--302.

\bibitem{MR2095975}
{\sc P.~Duchon, P.~Flajolet, G.~Louchard, and G.~Schaeffer}, {\em Boltzmann
  samplers for the random generation of combinatorial structures}, Combin.
  Probab. Comput., 13 (2004), pp.~577--625.

\bibitem{MR714482}
{\sc P.~Embrechts}, {\em The asymptotic behaviour of series and power series
  with positive coefficients}, Med. Konink. Acad. Wetensch. Belgi\"e, 45
  (1983), pp.~41--61.

\bibitem{MR772907}
{\sc P.~Embrechts and E.~Omey}, {\em Functions of power series}, Yokohama Math.
  J., 32 (1984), pp.~77--88.

\bibitem{MR2453776}
{\sc M.~M. Erlihson and B.~L. Granovsky}, {\em Limit shapes of {G}ibbs
  distributions on the set of integer partitions: the expansive case}, Ann.
  Inst. Henri Poincar\'e Probab. Stat., 44 (2008), pp.~915--945.

\bibitem{MR2498128}
{\sc P.~Flajolet, {\'E}.~Fusy, and C.~Pivoteau}, {\em Boltzmann sampling of
  unlabelled structures}, in Proceedings of the {N}inth {W}orkshop on
  {A}lgorithm {E}ngineering and {E}xperiments and the {F}ourth {W}orkshop on
  {A}nalytic {A}lgorithmics and {C}ombinatorics, SIAM, Philadelphia, PA, 2007,
  pp.~201--211.

\bibitem{MR2483235}
{\sc P.~Flajolet and R.~Sedgewick}, {\em Analytic combinatorics}, Cambridge
  University Press, Cambridge, 2009.

\bibitem{MR3097424}
{\sc S.~Foss, D.~Korshunov, and S.~Zachary}, {\em An introduction to
  heavy-tailed and subexponential distributions}, Springer Series in Operations
  Research and Financial Engineering, Springer, New York, second~ed., 2013.

\bibitem{MR633783}
{\sc A.~Joyal}, {\em Une th\'eorie combinatoire des s\'eries formelles}, Adv.
  in Math., 42 (1981), pp.~1--82.

\bibitem{MR2418771}
{\sc C.~McDiarmid}, {\em Random graphs on surfaces}, J. Combin. Theory Ser. B,
  98 (2008), pp.~778--797.

\bibitem{MR2507738}
\leavevmode\vrule height 2pt depth -1.6pt width 23pt, {\em Random graphs from a
  minor-closed class}, Combin. Probab. Comput., 18 (2009), pp.~583--599.

\bibitem{MR1662783}
{\sc L.~Mutafchiev}, {\em The largest tree in certain models of random
  forests}, in Proceedings of the {E}ighth {I}nternational {C}onference
  ``{R}andom {S}tructures and {A}lgorithms'' ({P}oznan, 1997), vol.~13, 1998,
  pp.~211--228.

\bibitem{MR2245368}
{\sc J.~Pitman}, {\em Combinatorial stochastic processes}, vol.~1875 of Lecture
  Notes in Mathematics, Springer-Verlag, Berlin, 2006.
\newblock Lectures from the 32nd Summer School on Probability Theory held in
  Saint-Flour, July 7--24, 2002, With a foreword by Jean Picard.

\bibitem{Mreplaceme}
{\sc B.~{Stufler}}, {\em {Gibbs partitions: the convergent case}}, ArXiv
  e-prints,  (2016).

\end{thebibliography}

\end{document}